\documentclass[12pt]{article}
\usepackage{graphicx}

\usepackage{times}

\usepackage{amsthm, amsmath, amssymb}
\usepackage{microtype}
\usepackage{enumerate}
\usepackage{hyperref}
\usepackage{todonotes}
\makeatletter

\newtheorem{thm}{Theorem}
\newtheorem{lemma}[thm]{Lemma}

\theoremstyle{definition}

\DeclareMathOperator{\rank}{rank}

\newsavebox{\fmbox}

\newcommand{\reals}{\ensuremath{\mathbb{R}}}

\newcommand{\naturals}{\ensuremath{\mathbb{N}}}

\newlength{\dyindent}
{\end{list}}

\newenvironment{dy*}{\begin{list}{}%
{\setlength{\leftmargin}{\dyindent}\setlength{\labelwidth}{\dyindent}%
\addtolength{\labelwidth}{-\labelsep}}%
\item}%
{\end{list}}

\newcommand{\mr}{\mbox{mr}}

\newcommand{\I}{\mathcal{I}}

\newcommand{\pin}{\mbox{pin}}

\title{On the inertia set of a signed graph with loops.}
\author{Marina Arav, Hein van der Holst\footnote{Corresponding author, E-mail: hvanderholst@gsu.edu}, John Sinkovic\\ 
Department of Mathematics and Statistics \\
Georgia State University \\
Atlanta, GA 30303, USA
}
\date{}
\begin{document}
\maketitle

\begin{abstract}
A signed graph is a pair $(G,\Sigma)$, where $G=(V,E)$ is a graph (in which parallel edges and loops are permitted) with $V=\{1,\ldots,n\}$ and $\Sigma\subseteq E$.  The edges in $\Sigma$ are called odd edges and the other edges of $E$ even. By $S(G,\Sigma)$ we denote the set of all symmetric $n\times n$ real matrices $A=[a_{i,j}]$ such that if $a_{i,j} < 0$, then there must be an even edge connecting $i$ and $j$; if $a_{i,j} > 0$, then there must be an odd edge connecting $i$ and $j$; and if $a_{i,j} = 0$, then either there must be an odd edge and an even edge connecting $i$ and $j$, or there are no edges connecting $i$ and $j$.  
(Here we allow $i=j$.)
For a symmetric real matrix $A$, the partial inertia of $A$ is the pair $(p,q)$, where $p$ and $q$ are the number of positive and negative eigenvalues of $A$, respectively. If $(G,\Sigma)$ is a signed graph, we define the \emph{inertia set} of $(G,\Sigma)$ as the set of the partial inertias of all matrices $A \in S(G,\Sigma)$.

In this paper, we present a formula that allows us to obtain the minimal elements of the inertia set of $(G,\Sigma)$ in case $(G,\Sigma)$ has a $1$-separation using the inertia sets of certain signed graphs associated to the $1$-separation.
\end{abstract}

\section*{Introduction}
A signed graph is a pair $(G,\Sigma)$, where $G=(V,E)$ is a graph (in which parallel edges and loops are permitted) with $V=\{1,\ldots,n\}$ and $\Sigma\subseteq E$.  The edges in $\Sigma$ are called odd edges and the other edges of $E$ even edges. By $S(G,\Sigma)$ we denote the set of all symmetric $n\times n$ real matrices $A=[a_{i,j}]$ such that 
\begin{itemize}
\item if $a_{i,j} < 0$, then there must be an even edge connecting $i$ and $j$,
\item if $a_{i,j} > 0$, then there must be an odd edge connecting $i$ and $j$, and
\item if $a_{i,j} = 0$, then either there must be an odd edge and an even edge connecting $i$ and $j$, or there are no edges connecting $i$ and $j$.  
\end{itemize}
Here we allow $i=j$, in which case loops might occur at vertex $i$.
For example the matrix
\begin{equation*}
A = \begin{bmatrix}
0 & 1 & 0 \\
1 & 0 & -2\\
0 & -2 & -3
\end{bmatrix}
\end{equation*}
belongs to $S(G,\Sigma)$, where $(G,\Sigma)$ is the signed graph shown in Figure~\ref{fig:example}. 
\begin{figure}
\begin{center}
\includegraphics[width=0.5\textwidth]{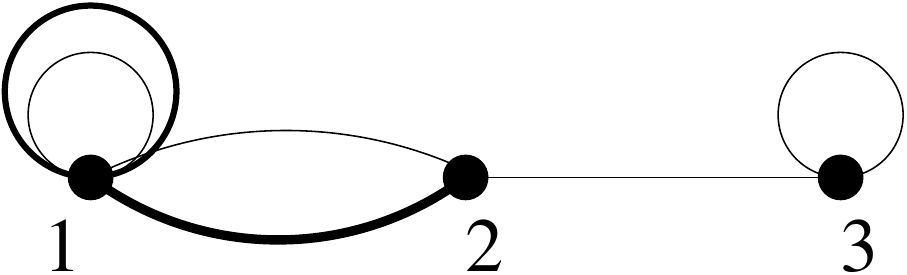}
\end{center}
\caption{Thick edges denote odd edges, while thin edges denote even edges.}\label{fig:example}
\end{figure}
For a symmetric real matrix $A$, the partial inertia of $A$, denoted by $\pin(A)$, is the pair $(p,q)$, where $p$ and $q$ are the number of positive and negative eigenvalues of $A$, respectively. If $(G,\Sigma)$ is a signed graph, we define the \emph{inertia set} of $(G,\Sigma)$ as the set $\{\pin(A)\mid A\in S(G,\Sigma)\}$; we denote the inertia set of $(G,\Sigma)$ by $\I(G,\Sigma)$. 

A \emph{separation} of a graph $G=(V,E)$ is a pair $(G_1,G_2)$ of subgraphs of $G$ such that $G_1\cup G_2 = G$ and $E(G_1)\cap E(G_2) = \emptyset$; its \emph{order} is the cardinality of $V(G_1)\cap V(G_2)$. If the order of a separation is $k$, we also say that $(G_1,G_2)$ is a $k$-separation. The notions of separations transfer without change to signed graphs.

Before presenting the formula, we need introduce some other notation. Let $\mathcal{S},\mathcal{R} \subseteq \naturals^2$; in this paper we include $0$ in the set $\naturals$. If for each $(p,q) \in \mathcal{S}$, there exists an $(r,s)\in \mathcal{R}$ such that $r\leq p$ and $s\leq q$, then we write $\mathcal{S}\leq \mathcal{R}$. If $\mathcal{S}\leq \mathcal{R}$ and $\mathcal{R}\leq \mathcal{S}$, then we write $\mathcal{R}\cong \mathcal{S}$. Call a pair $(p,q)\in \mathcal{S}$ \emph{minimal} if there is no pair $(r,s) \in \mathcal{S}$ with $r\leq p$, $s\leq q$ and $(p,q)\not=(r,s)$. Then $\mathcal{R}\cong \mathcal{S}$ if and only if every minimal pair in $\mathcal{R}$ also belongs to $\mathcal{S}$ and every minimal pair in $\mathcal{S}$ also belongs to $\mathcal{R}$. By $\mathcal{R}+\mathcal{S}$ we denote the set $\{(p_1+p_2,q_1+q_2)~\mid~(p_1,q_1)\in \mathcal{R}, (p_2,q_2)\in \mathcal{S}\}$.


Let $[(G_1,\Sigma_1), (G_2,\Sigma)]$ be a $1$-separation of a signed graph $(G,\Sigma)$ and let $v$ be the vertex in $V(G_1)\cap V(G_2)$.  For $i=1,2$, let $(G_i,\Sigma_i)_E$ and $(G_i,\Sigma_i)_O$ be the signed graphs obtained from $(G_i, \Sigma_i)$ by adding at $v$ an even loop and an odd loop, respectively. For $i=1,2$, let $(G_i,\Sigma_i)-v$ be the signed graph obtained from $(G_i,\Sigma)$ by deleting vertex $v$.
In this paper, we prove that the following formula holds:
\begin{equation}\label{mainformula}
\begin{split}
\I(G,\Sigma)\cong & [\I((G_1,\Sigma_1)-v)+ \I((G_2,\Sigma_2)-v)+ \{(1,1)\}]\\
 & \cup [\I(G_1,\Sigma_1) + \I(G_2,\Sigma_2)]\\
& \cup [\I(G_1,\Sigma_1)_E +\I(G_2,\Sigma_2)_O]\\
& \cup [\I(G_1,\Sigma_1)_O +\I(G_2,\Sigma_2)_E].
\end{split}
\end{equation}

If $(G_1,G_2)$ is a $1$-separation of a graph $G$, then 
Formula~(\ref{mainformula}) is analogous to the formula for the inertia set of $1$-sums of graphs. This formula was found by Barrett, Hall, and Loewy \cite{MR2547901}. For a graph $G=(V,E)$ with $V=\{1,\ldots,n\}$ (in which no parallel edges and loops are permitted), $S(G)$ demotes the set of all symmetric $n\times n$ real matrices $A=[a_{i,j}]$ such that $a_{i,j}\not=0$, $i\not=j$ if and only if $i$ and $j$ are adjacent. The inertia set of $G$ is $\{\pin(A)~\mid~A\in S(G)\}$ and is denoted by $\I(G)$. If $(G_1,G_2)$ is a $1$-separation of the graph $G$ and $\{v\} = V(G_1)\cap V(G_2)$, then $G$ is also called a $1$-sum of $G_1$ and $G_2$ at $v$. If $\mathcal{R}\subseteq \naturals^2$ and $n \in \naturals$, then $[\mathcal{R}]_n$ denotes the subset of $\mathcal{R}$ consisting of all pairs $(p,q)$ with $p+q\leq n$. Barrett, Hall, and Loewy proved that $\I(G) = [\I(G_1-v) + \I(G_2) + \{(1,1)\}]_n\cup [\I(G_1)+\I(G_2)]_n$. 

An important lemma about the inertia set in the case of graphs is the Northeast Lemma. This lemma says that if $G$ has order $n$ and $(p,q)\in \I(G)$ satisfies $p+q<n$, then also $(r,s) \in \I(G)$ for any $(r,s)$ satisfying $(p,q)\leq (r,s)$ and $r+s\leq n$. We note that, contrary to the graphical case, in the case of signed graphs in which parallel edges and loops are permitted, the Northeast lemma does not always hold. For example, if $(G,\emptyset)$ is the signed graph with exactly one vertex and no edges, then $(0,0)\in \I(G,\emptyset)$, while $(1,0),(0,1)\not\in \I(G,\emptyset)$.

The minimum rank of a graph $G$ is $\min\{\rank(A)~\mid~A\in S(G)\}$ and is denoted by $\mr(G)$. The inertia set of a graph $G$ includes the minimum rank of $G$, since $\mr(G) = \min\{p+q~\mid~(p,q) \in \I(G)\}$. Mikkelson~\cite{mikkelson2008} gives a formula for the minimum rank of $1$-sums of graphs that allow loops. In this case, a diagonal entry of a matrix corresponding to the graphs is zero or nonzero as to whether there is no loop or a loop at the corresponding vertex. See Fallat and Hogben \cite{FalHog2007} for a survey on the minimum ranks of graphs. 

In this paper, we allow matrices to have zero rows or zero columns. For example, if $A$ is $0\times n$ and $B$ is $n\times m$, then $AB$ is $0\times m$. If $A$ is $0\times n$, then each vector $x\in \mathbb{R}^n$ belongs to $\ker(A)$. If $k=0$, then $I_k$ denotes the $0\times 0$ matrix, while if $k>0$, then $I_k$ denotes the $k\times k$ identity matrix.

\section{Arrows on symmetric matrices}

If $A$ and $B$ are symmetric real matrices, we write 
\begin{equation*}
A\to B
\end{equation*}
if there exists a real matrix $P$ such that $P^T A P=B$.  It is clear that if $A\to B$ and $B\to C$, then $A\to C$. If for two symmetric real matrices $A$ and $B$, $A\to B$ and $B\to A$, then we write
\begin{equation*}
A\leftrightarrow B.
\end{equation*}

The following lemma is a variant of Sylvester's Law of Inertia (see, for example, Theorem 20.3 in \cite{Dym2007}).
\begin{lemma}\label{lem:diagonal}
Let $A$ and $B$ be symmetric real matrices. If $A\rightarrow B$ and $A$ has $p$ positive and $q$ negative eigenvalues, then $B$ has at most $p$ positive and at most $q$ negative eigenvalues. If $A\leftrightarrow B$, then $A$ and $B$ have the same number of positive and the same number of negative eigenvalues.
\end{lemma}

If $\mathcal{A}$ and $\mathcal{B}$ are sets of symmetric real matrices, we write 
\begin{equation*}
\mathcal{A}\to\mathcal{B}
\end{equation*}
if for every matrix $A\in \mathcal{A}$, there exists a matrix $B\in \mathcal{B}$ such that $A\to B$. 
If $\mathcal{A}$ consists of a single matrix $A$, then for $\mathcal{A}\to \mathcal{B}$ we also write $A\to \mathcal{B}$.
We write
\begin{equation*}
\mathcal{A}\leftrightarrow \mathcal{B}
\end{equation*}
if $\mathcal{A}\to\mathcal{B}$ and $\mathcal{B}\to\mathcal{A}$. 

For a set of symmetric $n\times n$ real matrices $\mathcal{A}$, we denote by $\I(\mathcal{A})$ the set $\{\pin(A)~\mid~A\in \mathcal{A}\}$. So $\I(G,\Sigma) = \I(S(G,\Sigma))$.

\begin{lemma}
Let $\mathcal{A}$ and $\mathcal{B}$ be sets of symmetric real matrices. If $\mathcal{A}\to \mathcal{B}$, then $\I(\mathcal{A})\leq \I(\mathcal{B})$. Consequently, if $\mathcal{A}\leftrightarrow \mathcal{B}$, then $\I(\mathcal{A})\cong\I(\mathcal{B})$.
\end{lemma}
\begin{proof}
Let $(p,q)\in \I(\mathcal{A})$. Then there exists a matrix $A\in \mathcal{A}$ with $p$ positive and $q$ negative eigenvalues. Since $\mathcal{A}\to \mathcal{B}$, there exists a matrix $B\in \mathcal{B}$ such that $A\to B$. By Lemma~\ref{lem:diagonal}, $B$ has at most $p$ positive and at most $q$ negative eigenvalues. Hence there exists a pair $(r,s) \in \I(\mathcal{B})$ with $(r,s)\leq (p,q)$. Thus, $\I(\mathcal{A})\leq \I(\mathcal{B})$.
\end{proof}

\section{Inequalities}
Let $0\leq k\leq m,n$ and let 
\begin{equation*}
A = \begin{bmatrix}
A_{1,1} & A_{1,2}\\
A_{2,1} & A_{2,2}
\end{bmatrix}\text{ and }
B = \begin{bmatrix}
B_{1,1} & B_{1,2}\\
B_{2,1} & B_{2,2} 
\end{bmatrix}
\end{equation*}
be symmetric $m\times m$ and $n\times n$ real matrices, respectively, where $A_{2,2}$ and $B_{1,1}$ are $k\times k$.
Then the $k$-subdirect sum of $A$ and $B$ (see \cite{FalJoh1999a}), which is denoted by $A\oplus_k B$, is the matrix
\begin{equation*}
A\oplus_k B = \begin{bmatrix}
A_{1,1} & A_{1,2} & 0\\
A_{2,1} & A_{2,2} + B_{1,1} & B_{1,2}\\
0 & B_{2,1} & B_{2,2}
\end{bmatrix}.
\end{equation*}
Let $0\leq k\leq\min(m,n)$. If $\mathcal{A}$ and $\mathcal{B}$ are sets of symmetric $m\times m$ and of symmetric $n\times n$ real matrices, respectively, then the $k$-subdirect sum of $\mathcal{A}$ and $\mathcal{B}$, denoted by $\mathcal{A}\oplus_k\mathcal{B}$ is the set of matrices
\begin{equation*}
\mathcal{A}\oplus_k\mathcal{B} = \{A\oplus_k B~\mid~A\in \mathcal{A}, B\in \mathcal{B}\}.
\end{equation*}
If $k=0$, we write $\mathcal{A}\oplus \mathcal{B}$ for $\mathcal{A}\oplus_k \mathcal{B}$, and if the set $\mathcal{B}$ consists of a single symmetric real matrix $B$, then we also write $\mathcal{A}\oplus B$ for $\mathcal{A}\oplus \mathcal{B}$.

Throughout the paper we denote by $H$ the matrix $\begin{bmatrix}
0 & 1\\
1 & 0
\end{bmatrix}$.

To prove Formula~\ref{mainformula}, it therefore suffices to prove that 
\begin{equation}\label{arrowformula}
\begin{split}
S(G,\Sigma)\leftrightarrow & [S((G_1,\Sigma_1)-v)\oplus S((G_2,\Sigma_2)-v)\oplus H]\\
 & \cup [S(G_1,\Sigma_1) \oplus S(G_2,\Sigma_2)]\\
& \cup [S(G_1,\Sigma_1)_E \oplus S(G_2,\Sigma_2)_O]\\
& \cup [S(G_1,\Sigma_1)_O \oplus S(G_2,\Sigma_2)_E].
\end{split}
\end{equation}

We now prove some lemmas and show one direction of the arrows. In the next section, we will finish the proof.

\begin{lemma}\label{lem:sumpin}
Let $0\leq k\leq \min(m,n)$. Let $A$ and $B$ be a symmetric $m\times m$ and a symmetric $n\times n$ real matrices, respectively. Then $A\oplus B\to A\oplus_k B$.
\end{lemma}
\begin{proof}
We may write
\begin{equation*}
A = \begin{bmatrix}
A_{1,1} & A_{1,2}\\
A_{2,1} & A_{2,2}
\end{bmatrix}
\end{equation*}
and
\begin{equation*}
B = \begin{bmatrix}
B_{2,2} & B_{2,3}\\
B_{3,2} & B_{3,3}
\end{bmatrix},
\end{equation*}
where $A_{2,2}$ and $B_{2,2}$ are $k\times k$ matrices.
Let
\begin{equation*}
P = \begin{bmatrix}
I_{m-k} & 0  & 0\\
0 & I_k  & 0\\
0 & I_k  & 0\\
0 & 0  & I_{n-k}
\end{bmatrix}.
\end{equation*}
Then
\begin{equation*}
P^T \begin{bmatrix}
A & 0\\
0 & B
\end{bmatrix} P = A\oplus_k B.
\end{equation*}
\end{proof}

From the previous lemma, we immediately obtain the following lemma.
\begin{lemma}\label{lem:1sepformulaonedirection}
Let $\mathcal{A}_1$ and $\mathcal{A}_2$ be sets of symmetric $m\times m$ real matrices and symmetric $n\times n$ real matrices, respectively. Then
\begin{equation*}
\mathcal{A}_1\oplus\mathcal{A}_2\to
\mathcal{A}_1\oplus_1\mathcal{A}_2
\end{equation*}
\end{lemma}

\begin{lemma}\label{lem:vertexdel}
Let $n>0$ and let $A = \begin{bmatrix}
A_{1,1} & A_{1,2}\\
A_{2,1} & a_{2,2}
\end{bmatrix}$ be a symmetric $n\times n$ real matrix, where $a_{2,2}$ is a scalar. Then $A_{1,1}\oplus H\rightarrow A\rightarrow A_{1,1}$.
\end{lemma}
\begin{proof}
To see that $A_{1,1}\oplus H\rightarrow A$, let
\begin{equation*}
P = \begin{bmatrix}
I_{n-1} & 0\\
0 & 1\\
A_{2,1} & a_{2,2}/2
\end{bmatrix}.
\end{equation*}
Then
\begin{equation*}
P^T
\begin{bmatrix}
A_{1,1} & 0 & 0\\
0 & 0 & 1\\
0 & 1 & 0
\end{bmatrix}P = A.
\end{equation*}

To see that $A\rightarrow A_{1,1}$, let
\begin{equation*}
P = \begin{bmatrix}
I_{n-1}\\
0
\end{bmatrix}.
\end{equation*}
Then $P^T A P = A_{1,1}$.
\end{proof}


If $A$ is a symmetric $n\times n$ matrix and $1\leq j\leq n$, then $A(j)$ denotes the principal submatrix of $A$ obtained by removing the $j$th row and column in $A$.
If $\mathcal{A}$ is a set of symmetric $n\times n$ matrices and $1\leq j\leq n$, then $\mathcal{A}(j)$ denotes the set $\{A(j)~\mid~A\in \mathcal{A}\}$.

From Lemma~\ref{lem:sumpin} and the previous lemma, we obtain the following lemma.
\begin{lemma}\label{lem:1sepformulaonedirectiondeleted}
Let $m$ and $n$ be positive integers, and let $\mathcal{A}_1$ and $\mathcal{A}_2$ be sets of symmetric $m\times m$ real matrices and symmetric $n\times n$ real matrices, respectively. Then
\begin{equation*}
\mathcal{A}_1(m)\oplus \mathcal{A}_2(1)\oplus H\to
\mathcal{A}_1\oplus_1\mathcal{A}_2
\end{equation*}
\end{lemma}

\begin{lemma}
Let $(G,\Sigma)$ be a signed graph which has a $1$-separation $[(G_1,\Sigma_1),(G_2,\Sigma_2)]$. Let $v$ be the vertex in $V(G_1)\cap V(G_2)$. Then each of the following holds:
\begin{enumerate}[(i)]
\item $S((G_1,\Sigma_1)-v)\oplus S((G_2,\Sigma_2)-v)\oplus H\to S(G,\Sigma)$, 
\item $S(G_1,\Sigma_1)\oplus S(G_2,\Sigma_2)\to S(G,\Sigma)$,
\item $S(G_1,\Sigma_1)_E\oplus S(G_2,\Sigma_2)_O\to S(G,\Sigma)$, and
\item $S(G_1,\Sigma_1)_O\oplus S(G_2,\Sigma_2)_E\to S(G,\Sigma)$.
\end{enumerate}
\end{lemma}

\begin{proof}
By Lemma~\ref{lem:1sepformulaonedirection},
\begin{equation*}
S(G_1,\Sigma_1)\oplus S(G_2,\Sigma_2)\oplus H\to S(G,\Sigma),
\end{equation*}
and by Lemma~\ref{lem:1sepformulaonedirectiondeleted},
\begin{equation*}
S((G_1,\Sigma_1)-v)\oplus S((G_2,\Sigma_2)-v)\oplus H\to S(G,\Sigma).
\end{equation*}

We now show that 
\begin{equation*}
S(G_1,\Sigma_1)_E\oplus S(G_2,\Sigma_2)_O\oplus H\to S(G,\Sigma).
\end{equation*}
Let $C = [c_{i,j}]\in S(G_1,\Sigma_1)_E$ and $D = [d_{i,j}]\in S(G_2,\Sigma_2)_O$. If $c_{v,v}d_{v,v} < 0$, then we can find a positive scalar $\alpha$ such that $A =[a_{i,j}]= \alpha C\oplus_1 D\in S(G,\Sigma)$. (To see this, notice that if $\alpha>0$ is sufficiently small, then $a_{v,v}$ has the same sign as $d_{v,v}$, while if $\alpha>0$ is sufficiently large, then $a_{v,v}$ has the same sign as $c_{v,v}$.) If $c_{v,v} > 0$ and $d_{v,v} > 0$, then $(G_1,\Sigma_1)$ has an odd loop at $v$, and so $(G,\Sigma)$ has an odd loop at $v$. Hence $C\oplus_1 D\in S(G,\Sigma)$. The case where $c_{v,v} < 0$ and $d_{v,v} < 0$ is similar. If $c_{v,v}=0$ and $d_{v,v} > 0$, then $(G_1,\Sigma_1)$ has an odd loop at $v$, and so $(G,\Sigma)$ has an odd loop at $v$. Hence $C\oplus_1 D\in S(G,\Sigma)$. If $c_{v,v} = 0$ and $d_{v,v} < 0$, then $(G_2,\Sigma_2)$ has an even loop at $v$, and so $(G,\Sigma)$ has an even loop. Hence $C\oplus_1 D\in S(G,\Sigma)$.
The cases where $d_{v,v}=0$ and either $c_{v,v} > 0$ or $c_{v,v}<0$ are similar. If $c_{v,v} = d_{v,v} = 0$, then $(S_1,\Sigma_1)$ has an odd loop at $v$ and $(S_2,\Sigma_2)$ has an even loop at $v$. Hence $(G,\Sigma)$ has an even and an odd loop at $v$, and so $C\oplus_1 D\in S(G,\Sigma)$. The proof that 
\begin{equation*}
S(G_1,\Sigma_1)_O\oplus S(G_2,\Sigma_2)_E\oplus H\to S(G,\Sigma)
\end{equation*}
is similar.
\end{proof}

\section{The inertia set of a signed graph with a 1-separation}\label{sec:minrank1sep}

In this section we finish the proof that Formula~(\ref{arrowformula}) is correct. This finishes the proof that Formula~(\ref{mainformula}) is correct. First we prove some lemmas.

\begin{lemma}\label{lem:vertexadd}
Let $n>0$ and let 
$\begin{bmatrix}
b_{1,1} & B_{1,2}\\
B_{2,1} & B_{2,2}
\end{bmatrix}$
be a symmetric $n\times n$ real matrix, where $b_{1,1}$ is a scalar. Let $a\not=0$. Then
\begin{equation*}
\begin{bmatrix}
0 & a & 0\\
a & b_{1,1} & B_{1,2}\\
0 & B_{2,1} & B_{2,2}
\end{bmatrix}\leftrightarrow H\oplus B_{2,2}.
\end{equation*}
\end{lemma}
\begin{proof}
Let
\begin{equation*}
P = \begin{bmatrix}
1/a & -b_{1,1}/(2a) & -(1/a)B_{1,2}\\
0 & 1 & 0\\
0 & 0 & I_{n-2}
\end{bmatrix}.
\end{equation*}
Then $P$ is invertible and 
\begin{equation*}
P^T \begin{bmatrix}
0 & a & 0\\
a & b_{1,1} & B_{1,2}\\
0 & B_{2,1} & B_{2,2}
\end{bmatrix} P = H\oplus B_{2,2},
\end{equation*}
hence the lemma follows.
\end{proof}

\begin{lemma}\label{lem:adjoin}
Let $\begin{bmatrix}
A & B\\
B^T & C
\end{bmatrix}$ be a symmetric $n\times n$ real matrix, where $A$ is a symmetric $k\times k$ real matrix. If $x\in \ker(C)$, then
\begin{equation*}
\begin{bmatrix}
0 & (Bx)^T & 0\\
Bx & A & B\\
0 & B^T & C
\end{bmatrix}\leftrightarrow
\begin{bmatrix}
A & B\\
B^T & C
\end{bmatrix}.
\end{equation*}
\end{lemma}
\begin{proof}
Let
\begin{equation*}
P = \begin{bmatrix}
0 & I_k & 0\\
x & 0 & I_{n-k}
\end{bmatrix}.
\end{equation*}
Then
\begin{equation*}
P^T \begin{bmatrix}
A & B\\
B^T & C
\end{bmatrix} P = \begin{bmatrix}
0 & (Bx)^T & 0\\
Bx & A & B\\
0 & B^T & C
\end{bmatrix},
\end{equation*}
so
\begin{equation*}
\begin{bmatrix}
A & B\\
B^T & C
\end{bmatrix}\rightarrow
\begin{bmatrix}
0 & (Bx)^T & 0\\
Bx & A & B\\
0 & B^T & C
\end{bmatrix}.
\end{equation*}
The other direction follows from Lemma~\ref{lem:vertexdel}.
\end{proof}

\begin{lemma}\label{lem:alternative}
Let $A = \begin{bmatrix}
A_{1,1} & A_{1,2} & 0\\
A_{2,1} & A_{2,2} & A_{2,3}\\
0 & A_{3,2} & A_{3,3}
\end{bmatrix}$ be a symmetric $n\times n$ matrix, where $A_{1,1}$ is $k\times k$ and $A_{2,2}$ is $m\times m$. Then at least one of the following holds:
\begin{enumerate}[(i)]
\item there exists a $k\times m$ matrix $Y$ such that
\begin{equation*}
\begin{bmatrix}
A_{1,1} & A_{1,2}\\
A_{2,1} & Y^T A_{1,1}Y
\end{bmatrix}\oplus
\begin{bmatrix}
A_{2,2} - Y^T A_{1,1}Y & A_{2,3}\\
A_{3,2} & A_{3,3}
\end{bmatrix}\leftrightarrow A,
\end{equation*}
or
\item there exists a nonzero vector $z \in \mathbb{R}^m$ such that
\begin{equation*}
\begin{bmatrix}
0 & 0 & z^T & 0\\
0 & A_{1,1} & A_{1,2} & 0\\
z & A_{2,1} & A_{2,2} & A_{2,3}\\
0 & 0 & A_{3,2} & A_{3,3}
\end{bmatrix}\leftrightarrow A.
\end{equation*}
\end{enumerate}
\end{lemma}
\begin{proof}
If $\begin{bmatrix}
A_{2,1} & A_{2,3}
\end{bmatrix}x = 0$ 
for each $x \in \ker(\begin{bmatrix}
                         A_{1,1} & 0\\
			 0 & A_{3,3}
                        \end{bmatrix})$, 
then there exists a matrix $m\times (n-m)$ matrix $W$ such that 
\begin{equation*}
W \begin{bmatrix}
A_{1,1} & 0\\
0 & A_{3,3}
\end{bmatrix} = \begin{bmatrix}
A_{2,1} & A_{2,3}
\end{bmatrix},
\end{equation*}
and so there exists an $k\times m$ matrix $Y$ such that $A_{1,1} Y = A_{1,2}$. Let
\begin{equation*}
P = \begin{bmatrix}
I_k & Y & -Y & 0\\
0 & 0 & I_m & 0\\
0 & 0 & 0 & I_{n-k-m}
\end{bmatrix}.
\end{equation*}
Then
\begin{equation*}
P^T A P = \begin{bmatrix}
A_{1,1} & A_{1,2}\\
A_{2,1} & Y^T A_{1,1}Y
\end{bmatrix}\oplus
\begin{bmatrix}
A_{2,2} - Y^T A_{1,1}Y & A_{2,3}\\
A_{3,2} & A_{3,3}
\end{bmatrix}.
\end{equation*}
From this and Lemma~\ref{lem:sumpin}, it follows that
\begin{equation*}
A\leftrightarrow \begin{bmatrix}
A_{1,1} & A_{1,2}\\
A_{2,1} & Y^T A_{1,1} Y
\end{bmatrix}\oplus
\begin{bmatrix}
A_{2,2} - Y^T A_{1,1}Y & A_{2,3}\\
A_{3,2} & A_{3,3}
\end{bmatrix}.
\end{equation*}

Thus we may assume that there exists a vector $u \in \ker(\begin{bmatrix}
                         A_{1,1} & 0\\
			 0 & A_{3,3}
                        \end{bmatrix})$ with 
\begin{equation*}
z := \begin{bmatrix}
      A_{2,1} & A_{2,3}
     \end{bmatrix}u
\end{equation*} 
nonzero. By Lemma~\ref{lem:adjoin}, 
\begin{equation*}
\begin{bmatrix}
0 & 0 & z^T & 0\\
0 & A_{1,1} & A_{1,2} & 0\\
z & A_{2,1} & A_{2,2} & A_{2,3}\\
0 & 0 & A_{3,2} & A_{3,3}
\end{bmatrix}\leftrightarrow A,
\end{equation*}
which finishes the proof.
\end{proof}

\begin{thm}\label{thm:general1sum}
Let $A = \begin{bmatrix}
A_{1,1} & A_{1,2} & 0\\
A_{2,1} & a_{2,2} & A_{2,3}\\
0 & A_{3,2} & A_{3,3}
\end{bmatrix}$
be a symmetric $n\times n$ real matrix, where $A_{1,1}$ is $k\times k$, $A_{3,3}$ is $(n-k-1)\times (n-k-1)$, and $a_{2,2}$ is a scalar. Then at least one of the following holds:
\begin{enumerate}[(i)]
\item there exists an $x\in \reals^k$ such that
\begin{equation*}
\begin{bmatrix}
A_{1,1} & A_{1,2}\\
A_{2,1} & x^T A_{1,1}x
\end{bmatrix}\oplus
\begin{bmatrix}
a_{2,2} - x^T A_{1,1}x & A_{2,3}\\
A_{3,2} & A_{3,3}
\end{bmatrix}\leftrightarrow A, 
\end{equation*}
or
\item $A_{1,1}\oplus A_{3,3}\oplus H\leftrightarrow  A$.
\end{enumerate}
\end{thm}
\begin{proof}
By  Lemma~\ref{lem:alternative} there exists a vector $x\in \reals^k$ such that
\begin{equation*}
A\leftrightarrow \begin{bmatrix}
A_{1,1} & A_{1,2}\\
A_{2,1} & x^T A_{1,1}x
\end{bmatrix}\oplus
\begin{bmatrix}
a_{2,2} - x^T A_{1,1}x & A_{2,3}\\
A_{3,2} & A_{3,3}
\end{bmatrix},
\end{equation*}
or there exists a nonzero scalar $z$ such that
\begin{equation*}
A\leftrightarrow\begin{bmatrix}
0 & 0 & z & 0\\
0 & A_{1,1} & A_{1,2} & 0\\
z & A_{2,1} & a_{2,2} & A_{2,3}\\
0 & 0 & A_{3,2} & A_{3,3}
\end{bmatrix}.
\end{equation*}
By Lemma~\ref{lem:vertexadd},
\begin{equation*}
\begin{bmatrix}
0 & 0 & z & 0\\
0 & A_{1,1} & A_{1,2} & 0\\
z & A_{2,1} & a_{2,2} & A_{2,3}\\
0 & 0 & A_{3,2} & A_{3,3}
\end{bmatrix}\leftrightarrow H\oplus A_{1,1}\oplus A_{3,3}.
\end{equation*}
\end{proof}

\begin{lemma}\label{lem:split1sep}
Let $[(G_1,\Sigma_1),(G_2,\Sigma_2)]$ be a $1$-separation of the signed graph $(G,\Sigma)$, and let $\{v\} = V(G_1)\cap V(G_2)$. If $A\in S(G,\Sigma)$, then there exist matrices $B = [b_{i,j}]\in S(G_1,\Sigma_1)$ and $C=[c_{i,j}]\in S(G_2,\Sigma_2)$ such that $A = B\oplus_1 C$.
\end{lemma}
\begin{proof}
Except for the entries $b_{v,v}$ and $c_{v,v}$, all other entries of $B$ and $C$ are determined by $A$.

Suppose $a_{v,v} > 0$. Then $(G,\Sigma)$ has an odd edge at vertex $k$. Hence $(G_1,\Sigma_1)$ or $(G_2,\Sigma_2)$ has an odd edge at vertex $v$; by symmetry, we may assume that $(G_1,\Sigma_1)$ has an odd edge at $v$. If $(G_2,\Sigma_2)$ has an odd edge at $v$, let $b_{v,v} = c_{v,v} = a_{v,v}/2$. Otherwise, $c_{v,v}\leq 0$ always. Then, we let $b_{v,v} = a_{v,v} - c_{v,v}$.
The case where $a_{v,v} < 0$ is similar. 

Suppose now that $a_{v,v} = 0$. If $(G,\Sigma)$ has no loops at $v$, then both $(G_1,\Sigma_1)$ and $(G_2,\Sigma_2)$ have no loops at $v$. Then $b_{v,v} = c_{v,v} = 0$. We now assume that $(G,\Sigma)$ has loops at $v$. Then there is at least one even and at least one odd loop at $v$. If $(G_1,\Sigma_1)$ has no even loops at $v$, then $(G_2,\Sigma_2)$ has an even loop at $v$. If $(G_1,\Sigma_1)$ has an odd loop at $v$, then we let $b_{v,v} = 1$ and $c_{v,v} = -1$. If $(G_1,\Sigma_1)$ has no odd loops at $v$, then $(G_2,\Sigma_2)$ has an odd loop at $v$. Then we let $b_{v,v} = c_{v,v} = 0$. The case where $(G_1,\Sigma_1)$ has no odd loops at $v$ is similar. So we may assume that $(G_1,\Sigma_1)$ and, by symmetry also $(G_2,\Sigma_2)$, have an even and an odd loop. Then we let $b_{v,v} = c_{v,v} = 0$.
\end{proof}

\begin{thm}
Let $(G,\Sigma)$ be a signed graph which has a $1$-separation $[(G_1,\Sigma_1),(G_2,\Sigma_2)]$. Let $v$ be the vertex in $V(G_1)\cap V(G_2)$. Then
\begin{equation}
\begin{split}
S(G,\Sigma)\leftrightarrow & [S((G_1,\Sigma_1)-v)\oplus S((G_2,\Sigma_2)-v)\oplus H]\\
 & \cup [S(G_1,\Sigma_1) \oplus S(G_2,\Sigma_2)]\\
& \cup [S(G_1,\Sigma_1)_E \oplus S(G_2,\Sigma_2)_O]\\
& \cup [S(G_1,\Sigma_1)_O \oplus S(G_2,\Sigma_2)_E].
\end{split}
\end{equation}
\end{thm}

\begin{proof}
By the previous section,
\begin{equation*}
\begin{split}
[S((G_1,\Sigma_1)-v)\oplus S((G_2,\Sigma_2)-v)\oplus H]\cup &\\
[S(G_1,\Sigma_1)\oplus S(G_2,\Sigma_2)]\cup &\\
[S(G_1,\Sigma_1)_E\oplus S(G_2,\Sigma_2)_O]\cup &\\
[S(G_1,\Sigma_1)_O\oplus S(G_2,\Sigma_2)_E] & \to S(G,\Sigma).
\end{split}
\end{equation*}
We now show that the converse direction also holds.

Let
\begin{equation*}
C = \begin{bmatrix}
C_{1,1} & C_{1,2} & 0\\
C_{2,1} & c_{2,2} & C_{2,3}\\
0 & C_{3,2} & C_{3,3}
\end{bmatrix}\in S(G,\Sigma).
\end{equation*}
Then, by Theorem~\ref{thm:general1sum}, at least one of the following holds:
\begin{enumerate}[(i)]
\item\label{item4} $C_{1,1} \oplus C_{3,3} \oplus H\leftrightarrow C$.
\item\label{item1} There exists a vector $x$ such that
\begin{equation*}
\begin{bmatrix}
C_{1,1} & C_{1,2}\\
C_{2,1} & x^T C_{1,1} x
\end{bmatrix}\oplus
\begin{bmatrix}
c_{2,2} - x^T C_{1,1} x & C_{2,3}\\
C_{3,2} & C_{3,3}
\end{bmatrix}\leftrightarrow C.
\end{equation*}
\end{enumerate}

Suppose first that $(\ref{item4})$ holds. Then
\begin{equation*}
C\to S((G_1,\Sigma_1)-v)\oplus S((G_2,\Sigma)-v)\oplus H.
\end{equation*}

Suppose now that $(\ref{item1})$ holds. By Lemma~\ref{lem:split1sep}, there exist matrices 
\begin{equation*}
B = \begin{bmatrix}
C_{1,1} & C_{1,2}\\
C_{2,1} & b
\end{bmatrix}\in S(G_1,\Sigma_1)\text{ and }
D = \begin{bmatrix}
d & C_{2,3}\\
C_{3,2} & C_{3,3}
\end{bmatrix}\in S(G_2,\Sigma_2)
\end{equation*}
such that $C = B\oplus_1 D$. (So $b+d=c_{2,2}$.)
If $x^T C_{1,1} x - b > 0$, then 
\begin{equation*}
\begin{bmatrix}
C_{1,1} & C_{1,2}\\
C_{2,1} & x^T C_{1,1} x
\end{bmatrix} = \begin{bmatrix}
C_{1,1} & C_{1,2}\\
C_{2,1} & b + (x^T C_{1,1} x - b)
\end{bmatrix}\in S(G_1,\Sigma_1)_O
\end{equation*}
and
\begin{equation*}
\begin{bmatrix}
c_{2,2} - x^T C_{1,1} x  & C_{2,3}\\
C_{3,2} & C_{3,3}
\end{bmatrix} = \begin{bmatrix}
d - (x^T C_{1,1} x - b) & C_{2,3}\\
C_{3,2} & C_{3,3}
\end{bmatrix}\in S(G_2,\Sigma_2)_E.
\end{equation*}
Hence $C\to S(G_1,\Sigma_1)_O\oplus S(G_2,\Sigma_2)_E$.
The cases where $x^T C_{1,1} x - b = 0$ and $x^T C_{1,1} x - b < 0$ are similar.

\end{proof}

\section{All terms are needed}
We now exhibit several examples of signed graphs illustrating that each term in Formula~(\ref{mainformula}) is needed.

To see that the term $\I((G_1,\Sigma_1)-v)+\I(G_2,\Sigma_2)-v)+\{(1,1)\}$ is needed in Formula~(\ref{mainformula}), let $(G,\Sigma)$ be the signed graph where $G$ is a $2$-path, all edges are odd, and none of the vertices has a loop. Let $(G_1,\Sigma_1)$ be the signed subgraph of $(G,\Sigma)$ consisting of one odd edge $e$, and let $(G_2,\Sigma_2)$ be the signed subgraph of $(G,\Sigma)$ consisting of the other odd edge. Then Formula~(\ref{mainformula}) shows that $(1,1)\in \I((G_1,\Sigma_1)-v)+\I((G_2,\Sigma_2)-v)+\{(1,1)\}$, while 
\begin{equation*}
\begin{split}
(0,0), (1,0), (0,1), (1,1)\not\in & [\I(G_1,\Sigma_1) + \I(G_2,\Sigma_2)] \\
\cup & [\I(G_1,\Sigma_1)_E +\I(G_2,\Sigma_2)_O]\\
\cup & [\I(G_1,\Sigma_1)_O +\I(G_2,\Sigma_2)_E].
\end{split}
\end{equation*}

To see that the term $\I(G_1,\Sigma_1)+\I(G_2,\Sigma_2)$ 
is needed in Formula~(\ref{mainformula}), let $(G,\Sigma)$ be the signed graph consisting of three isolated vertices and no edges. Let $(G_1,\Sigma_1)$ and $(G_2,\Sigma_2)$ be distinct signed subgraphs, each consisting of two vertices. Then $(0,0) \in \I(G_1,\Sigma_1)+\I(G_2,\Sigma_2)$, while 
\begin{equation*}
\begin{split}
(0,0)\not\in & [\I((G_1,\Sigma_1)-v) + \I((G_2,\Sigma_2)-v) + \{(1,1)\}] \\
\cup & [\I(G_1,\Sigma_1)_E +\I(G_2,\Sigma_2)_O]\\
\cup & [\I(G_1,\Sigma_1)_O +\I(G_2,\Sigma_2)_E].
\end{split}
\end{equation*}

To see that the term $\I(G_1,\Sigma_1)_E+\I(G_2,\Sigma_2)_O$ 
is needed in Formula~(\ref{mainformula}), let $(G,\Sigma)$ be the signed graph that is a $2$-path, all edges of which are odd, and with at each vertex an odd loop. Let $(G_1,\Sigma_1)$ be a signed subgraph consisting of one odd edge connecting distinct vertices and the two odd loops at these vertices. Let $(G_2,\Sigma_2)$ be the signed subgraph consisting of the remaining edges and the ends these edges.
Then $(2,0)\in \I(G_1,\Sigma_1)_E+\I(G_2,\Sigma_2)_O$, while 
\begin{equation*}
\begin{split}
(0,0), (1,0), (2,0)\not\in & [\I((G_1,\Sigma_1)-v) + \I((G_2,\Sigma_2)-v) + \{(1,1)\}]\\
\cup & [\I(G_1,\Sigma_1) + \I(G_2,\Sigma_2)]\\
\cup & [\I(G_1,\Sigma_1)_O + \I(G_2,\Sigma_2)_E].
\end{split}
\end{equation*}
Switching the roles of odd and even shows that the term $\I(G_1,\Sigma_1)_O+\I(G_2,\Sigma_2)_E$ is needed in Formula~(\ref{mainformula}).

 \newcommand{\noopsort}[1]{}


\begin{thebibliography}{1}

\bibitem{MR2547901}
W.~Barrett, H.~T. Hall, and R.~Loewy.
\newblock The inverse inertia problem for graphs: cut vertices, trees, and a
  counterexample.
\newblock {\em Linear Algebra Appl.}, 431(8):1147--1191, 2009.

\bibitem{Dym2007}
H.~Dym.
\newblock {\em Linear Algebra in Action}, volume~78 of {\em Graduate Studies in
  Mathematics}.
\newblock American Mathematical Society, 2007.

\bibitem{FalHog2007}
S.~Fallat and L.~Hogben.
\newblock The minimum rank of symmetric matrices described by a graph: A
  survey.
\newblock {\em Linear Algebra Appl.}, 426(2--3):558--582, 2007.

\bibitem{FalJoh1999a}
S.~M. Fallat and C.~R. Johnson.
\newblock Sub-direct sums and positivity classes of matrices.
\newblock {\em Linear Algebra Appl.}, 288:149--173, February 1999.

\bibitem{mikkelson2008}
R.~C. Mikkelson.
\newblock {\em Minimum Rank of Graphs that Allow Loops}.
\newblock PhD thesis, Iowa State University, 2008.

\end{thebibliography}
\end{document}